\documentclass[12 pt]{article}
\usepackage{pdfsync}
\usepackage{hyperref}

\usepackage{xcolor}

\usepackage{pdfcomment}
\usepackage{todonotes}

\usepackage{amsfonts,amsmath,amssymb,amsthm}
\usepackage{mathtools} 

\usepackage{mathabx}
\usepackage{bbm}
\usepackage{mathabx} 

\usepackage{comment}
\usepackage{enumitem}

\usepackage{pgfplots}
\usepackage{tikz}
\usetikzlibrary{arrows}

\usepackage{import} 

\usepackage[mathscr]{euscript}

\newcommand{\B}{\mathcal{B}}
\newcommand{\M}{\mathcal{M}}

\newcommand{\N}{\mathbb{N}}
\newcommand{\Z}{\mathbb{Z}}

\newcommand{\R}{\mathbb{R}}

\newcommand{\T}{\mathbb{T}}

\newtheorem{theorem}{Theorem}
\numberwithin{theorem}{section} 

\newtheorem{lemma}[theorem]{Lemma}

\newtheorem{corollary}[theorem]{Corollary}
\theoremstyle{definition}
\newtheorem{definition}[theorem]{Definition}
\newtheorem{example}[theorem]{Example}

\newtheorem{remark}[theorem]{Remark}
\numberwithin{equation}{section}

\usepackage{scalerel,stackengine}
\stackMath
\newcommand\reallywidehat[1]{%
\savestack{\tmpbox}{\stretchto{%
  \scaleto{%
    \scalerel*[\widthof{\ensuremath{#1}}]{\kern.1pt\mathchar"0362\kern.1pt}%
    {\rule{0ex}{\textheight}}
  }{\textheight}%
}{2.4ex}}%
\stackon[-6.9pt]{#1}{\tmpbox}%
}

\begin{document}
\title{A Marcinkiewicz Multiplier Theory for Schur Multipliers}

\author {Chian Yeong  Chuah, Zhen-Chuan  Liu, Tao Mei}

\date{}
\maketitle
\begin{abstract}
 We prove a Marcinkiewicz type multiplier theory for the boundedness of Schur multipliers on the Schatten $p$-classes. This generalizes a previous result of J. Bourgain for Toeplitz type Schur multipliers and complements a recent result by J. Conde et al. (\cite{CGPT22a}). As a corollary, we obtain a  new unconditional decomposition for the Schatten $p$-classes, $1<p<\infty$. We extend our main result to the $\Z^d$ and $\R^d$ cases, and include an operator-valued  version of it using Pisier's non-commutative $L^\infty(\ell_1)$-norm.
\end{abstract}


\section{Introduction}

 Let $A\in  B(H)$ be a bounded operator on a (separable) Hilbert space $H$. We can write $A$ in its matrix representation
   $$A=(a_{k,j})_{k,j\in\Z}$$ with $a_{k,j}=\langle Ae_k,e_j\rangle $ for a given orthonormal basis $\{e_k\}_{k\in \Z}$ of $H$.
Given a bounded function $m$ on ${\Bbb Z}\times {\Bbb Z}$, 
 we call the map 
 \begin{eqnarray}
 M_m: (a_{k,j})\mapsto (m(k,j)a_{k,j})\label{sm}
 \end{eqnarray}
  a Schur multiplier with symbol $m$.
The study of the boundedness of Schur multipliers with respect to the Schatten $p$-norms   has a rich history (\cite{Be77, Ar82, BoFe84, BeGi94,  Pi98, Ha99, ClPaSuWi00, Pi01, AlPe02, DoGi05}). The recent study of non-commutative analysis on the approximation properties of operator functions and operator algebras  (\cite{HaStSz10, NeRi11, CaSa15, 
PST15, PST17, LaSa18, CPSZ19, PaRiSa22, MRX22, CGPT22a} etc.), especially Lafforgue/de La Salle's recent work (\cite{LaSa11}) on the approximation property of higher rank Lie groups redraws  a lot of attention to the  boundedness of Schur multipliers  for the case where  $1<p\neq2<\infty$.  
Conde/Gonz\'alez/Parcet/Tablate   recently  proved   a H\"ormander-Mikhlin type   Schur multiplier theory for   $S^p, 1<p\neq 2<\infty$ in their remarkable work \cite{CGPT22a}. 

In this article, we prove a Marcinkiewicz type Schur multiplier theory. H\"ormander-Mikhlin  type multipliers and  Marcinkiewicz  type multipliers are rooted from classical Fourier analysis. Like their counterpart in Fourier analysis, Marcinkiewicz  type Schur multipliers  are a larger class of multipliers and their $p$-boundedness is more subtle as shown in the  work of Tao and Wright (\cite{TaWr01}) that the $L^p$-bounds of Marcinkiewicz  Fourier multipliers  is of order $p^{3/2}$ as $p\rightarrow\infty$. 

When $m$ is of Toeplitz type, i.e. $m(k,j)=\dot m(k-j)$ for some function $\dot m:\Z\to {\Bbb C}$, one may apply a well-known transference method and obtain  bounded Schur multipliers from the classical Fourier multiplier theory. 
 J. Bourgain (\cite[Theorem 4, Corollary 20]{Bou86})'s work of scalar-valued Fourier multipliers acting on Schatten $p$-valued functions  implies that, the   following  Marcinkiewicz type  condition is sufficient for the boundedness of $M_m$ on the Schatten $p$-classes for all $1<p<\infty$,
\begin{eqnarray}\label{mdot}
\sum_{2^{n-1}\leq |k|< 2^{n}}|\dot m(k+1)-\dot m(k)| <C,
\end{eqnarray}
for all $n\in \N$.  
Let $\dot m_\varepsilon(k)=\varepsilon_n$ for $2^{n-1}\leq |k|< 2^{n}$. Then the associated multiplier $M_{m_\varepsilon}$ is bounded for any  sequence $\varepsilon_n=\pm1$.

To extend Bourgain's result to general non-Toeplitz type Schur multipliers, one may  ask whether the    condition that 
\begin{eqnarray}
  \sum_{2^{n-1}\leq |k|< 2^{n}}|m(k+j+1,j)-m(k+j,j)| <C \label{gMarc}
\end{eqnarray} 
for all $n\in{\Bbb N}, j\in \Z$, implies the $S^p$ boundedness of general Schur multipliers. The answer is yes if $m$ is Toeplitz since condition \eqref{gMarc} reduces to Bourgain's condition \eqref{mdot} in that case. 
The answer would be yes for general Schur multipliers  as well if   the family $M_{m_\varepsilon}$, defined after condition \eqref{mdot}, is $R$-bounded for any family of sequences $\varepsilon=(\varepsilon_k)_k$ valued in $\{\pm1\}$. This implication was proved in  the work of Berkson/Gilliespie (\cite{BeGi94,DoGi05}) and Clement/Pagter/Sukochev/Witvliet (\cite{ClPaSuWi00}), in which they studied the connection between the vector-valued Littlewood-Paley theory  and the Marcinkiewicz multiplier Theory. 
We show in Section \ref{sec:Counterex}  that this is not true in general and the condition \eqref{gMarc} is not sufficient for the $S^p$-boundedness of the associated Schur multiplier.  \footnote{This also shows that the family of ``Littlewood-Paley operators" $M_{m_\varepsilon}$mentioned above is NOT $R$-bounded.}

The main result of this article is the following.
 \begin{theorem} \label{thm:main}
\vskip.1cm
 $M_m$ defined in (\ref{sm}) extends to a   bounded map on the Schatten $p$-classes $S^p$ for all $1<p<\infty$ with bounds $\lesssim (\frac{p^2}{p-1})^3$, if  $m$ is bounded and there exists a constant $C$ such that  
 \begin{eqnarray}
  \sum_{2^{n-1}\leq |k|< 2^{n}}|m(k+j+1,j)-m(k+j,j)| <C \label{Marc}\\
 \sum_{2^{n-1}\leq |k|<2^{n}}|m(j,k+j+1)-m(j,k+j)| <C \label{Marr},
\end{eqnarray} 
for all $n\in{\Bbb N}, j\in \Z$.
 \end{theorem}
 The writing of this article was motivated by the recent article \cite{CGPT22a}, although the third author had known Theorem \ref{thm:main} previously.  The authors of \cite{CGPT22a} further studied Schatten-$p$-classes indexed in $d$-dimensional Euclidean spaces aiming possible applications to the approximation properties of  higher rank Lie groups. Following this trend, we extend Theorem \ref{thm:main}  to the higher dimensional cases as well.  The proof of Theorem \ref{thm:main} relies on the crucial property  that a Schur multiplier is an  operator-valued Fourier multiplier multiplying from the left, and is simultaneously  an  operator-valued Fourier multiplier multiplying from the right. This property was  already used by the authors of \cite{CGPT22a} in proving a H\"ormander-Mikhlin type criterion for  the boundedness of Schur multipliers.

Theorem \ref{thm:main} implies   a new unconditional decomposition for Schatten $p$ classes. For $(n,\ell)\in {\Bbb Z}\times \Z$, let $E_{0,\ell}= \{(\ell,\ell)\}\subset {\Bbb Z}\times \Z$. Let
 $$E_{n,\ell}= \{(k,j)\in {\Bbb Z}\times\Z; 2^{n-1}\leq k-j< 2^n , \ell2^{n}\leq k< (\ell+1)2^{n}\}$$ for $n>0$, and $$E_{n,\ell}= \{(k,j)\in {\Bbb Z}\times\Z;   -2^{|n|}<k-j\leq -2^{|n|-1}, \ell2^{|n|}\leq k< (\ell+1)2^{|n|}\}$$ for $n<0$. We then have the decomposition 
 $${\Bbb Z}\times\Z=\bigcup_{(n,\ell)\in\Z\times\Z} E_{n,\ell}.$$
 Let $P_{n,\ell}$ be the projection onto the $span\{e_{k,j}, (k,j)\in E_{n,\ell}\}$. It is easy to see that $\sum_{n,\ell} \varepsilon_{n,\ell} P_{n,\ell}$ is a Schur multiplier satisfying the assumptions of Theorem \ref{thm:main} for any bounded sequence $\varepsilon_{n,\ell}$. We then obtain an unconditional decomposition of $S^p$.
 
\begin{corollary} \label{cor:main} $\sum_{n,\ell} \varepsilon_{n,\ell} P_{n,\ell}$ extends to a   bounded map on $S^p$ for all $1<p<\infty$ for any bounded sequence $\varepsilon_{n,\ell}$.
\end{corollary}


We will prove Theorem \ref{thm:main}  in Section \ref{sec:ProofofThm}. We will explain how to extend Theorem \ref{thm:main} to the higher dimensional case  in Section \ref{sec:highDimension} and explain that the ball-type Schur multipliers  remain bounded on $S^p(\ell^2(\Z^d))$  for $d>1$ (  Example \ref{example4.4}), contrary to the behavior of Fourier multipliers. We  will show that the condition \eqref{Marc} alone is not sufficient in Section \ref{sec:Counterex}, and explain an operator-valued version of Theorem \ref{thm:main} in Section \ref{sec:OperatorValued}.
\section{Preliminaries}
Given $d\in \N$, denote by $\mathcal{B}(\ell^2(\Z^d))$  the set of bounded linear operators on $\ell^2(\Z^d)$. We represent the operator $A\in \mathcal{B}(\ell^2(\Z^d))$ as $A=(a_{i,j})_{(i,j)\in \Z^d\times \Z^d}$ with $a_{i,j}=\langle Ae_i,e_j\rangle $ for the canonical basis $\{e_i\}$ of $\ell^2(\Z^d)$. Given a bounded function $m$ on $\Z^d\times \Z^d$, 
the associated Schur multiplier 
\[
	M_m(A)=(m(i,j)a_{ij})
\]
extends to a bounded  operator on the Hilbert-Schmidt class $S^2(\ell^2(\Z^d))$. 
We call $m$ the symbol of $M_m$. Recall that the Schatten $p$-class $S^p, 1\leq p<\infty$ is the collection of all compact operators $A$ with a finite  Schatten-$p$ norm, which  is defined as
\begin{eqnarray}
\|A\|_p&=& \left(tr \left[ (A^*A)^{\frac p2} \right] \right)^{\frac1p}= \left( \sum_{i}s_i^p \right)^\frac1p, \label{sp}
 \end{eqnarray} 
 for $1\leq p<\infty$, where $s_i$ is the $i$-th singular value of $A$. 
The Schatten $p$ norm is unitary invariant and does not depend on the choice of the orthonormal basis. The Schatten-class $S^p, 1\leq p<\infty$ and $B(\ell^2(\Z^d))$ share many similar properties to $\ell^p, 1\leq p\leq \infty$. In particular,  the dual space of  $S^1$ (resp. $S^p, 1<p<\infty$) is  isomorphic to $B(\ell^2)$ (resp. $S^{\frac p{p-1}})$. The family forms an interpolation scale that, for $1<p<\infty$, $$[B(\ell^2),S^1]_{\frac{1}p}=S^p.$$ However, $S^p$ does not admit an unconditional basis whenever $p\neq 2$.
 We will prove that, for  $m$ satisfying additional conditions (\ref{Marc}) and (\ref{Marr}),    $M_m$ extends to a bounded map on  $S^p$  for all $1<p<\infty$,  which immediately implies the unconditional decomposition for $S^p$ as stated in Corollary \ref{cor:main}.



Given $d\in {\Bbb N}$, we denote by $L^p(\T^d;S^p)$ the space of $S^p$-valued Bochner integrable function $f$, such that 
$$\|f\|_{L^p}= \left( tr \left[ \int_{[0,1)^d}|f|^p(z) d\theta \right] \right)^\frac1p<\infty,$$
Here we identify $z=e^{i2\pi \theta}$ with $\theta \in [0,1)^d$, and $|f|^p=(f^{*} f)^\frac p2$ is defined via the functional calculus. 

 For $f\in L^p(\T^d;S^p)$ with $ 1<p<\infty$, we have the 
Fourier expansion as follows: 
$$f(z)\sim\sum _{k\in \Z^d} \hat{f}(k) z^k,$$ with $\hat f(k)=\int_{[0,1]^d}f(z) \bar z^{k} d\theta \in S^p$. Given  $R$ a finite subset of ${\Bbb Z}^d$, denote by $S_{R} f$ the   partial Fourier sum  
\begin{eqnarray}
	S_{R} f(z)= \sum_{k\in R} \hat{f}(k) z^k.\label{SR}
\end{eqnarray}

Choose $\delta\in C^{\infty}(\R)$ such that   $0\leq\delta\leq 1$, $supp(\delta)\subset [-2\sqrt d ,-\frac{1}{4}]\cup[\frac{1}{4},2\sqrt d]$ and
 $\delta(x)=1$ when $\frac12\leq |x|\leq \sqrt d$. For $j\geq0$, define $\delta_j(x):=\delta(2^{-j}x)$. For $f\in L^2(\T^d;S^2)$, we define
 \[
 S_jf(z)=\sum_{k\in\Z^d} \delta_j(|k|_{2}) \hat{f}(k) z^k.\]
We denote by $|k|_2$ the $\ell_2$ norm of $k\in {\Z}^d$ in the formula above and will denote by $|k|_\infty$ the $\ell_\infty$ norm of $k$. Let $(E_j)_{j\geq 0}$ be the cubes with squared holes in $\Z^d$  given by 
\begin{equation}\label{eq:dyadic}
E_j =
\begin{cases}
		\{k\in\Z^d: 2^{j-1}\leq  |k|_\infty< 2^j\} & j>0,\\
		\{0\} & j=0.
	\end{cases}
\end{equation} Note our construction implies $ S_{E_j}S_j=S_{E_j}$ which we will need later.

For a sequence $(f_k)$ in $L^p( {\Bbb T}^d;S^p)$, we use the classical notation 
$$\|(f_k)\|_{L^p(\ell_2^{c})}=\|(\sum_k|f_k|^2)^\frac12\|_{L^p(\T^d;S^p)},\quad \|(f_k)\|_{L^p( \ell_2^{r})}=\|(\sum_k|f^*_k|^2)^\frac12\|_{L^p(\T^d;S^p)},$$ and 
$$\|(f_k)\|_{L^p(\ell_2)}=\left\{\begin{array}{ll} \max\{
\|(f_k)\|_{L^p(\ell_2^{c})}, \|(f^*_k)\|_{L^p(\ell_2^{c})}\}&
\textrm{if } 2\leq p\leq \infty\\ \inf_{ y_k+z_k=f_k}
\|(y_k)\|_{L^p( \ell_2^{c})}+ \|(z_k)\|_{L^p( \ell_2^{r})}&
\textrm{if } 0<p<2\end{array}\right..$$
The above definition is justified by the following noncommutative Khintchine inequality:

\begin{lemma} \label{lemmakhin}(\cite{L86}, \cite{LP91}) Let $(\varepsilon_k)$ be a sequence of  independent Rademacher random variables. Then for $1\leq p<\infty$, 
\begin{eqnarray}
\alpha_p^{-1}E_\varepsilon\|\sum_k \varepsilon_k\otimes f_k\|_{L^p(\T^d;S^p)}&\leq& \|(f_k)\|_{L^p( \ell_2)}\leq \beta_p E_\varepsilon \|\sum_k \varepsilon_k\otimes f_k\|_{L^p(\T^d;S^p)}. \label{khin}
\end{eqnarray}
\end{lemma}
 The  optimal constant $\beta_p$
is no greater than $\sqrt 3$ for $1\leq p\leq 2$ and is $1$ for $p\geq 2$ (see
\cite{HM07}). $\alpha_p$ is $1$ for $1\leq p\leq 2$ and is of order
$\sqrt p$ as $p\rightarrow \infty$. 
(\ref{khin}) was pushed further to
the case where $0<p<1$ by Pisier and Ricard (see \cite{PR14}).

\medskip

 We will need   the following noncommutative Littlewood-Paley theorem on $\Z^d$. 
\begin{lemma}\label{thm:Littwood-P}
 There is a constant $C_d>0$, that depends only on $d$, such that, \begin{eqnarray}
&&\|(S_j f )_{j\geq0}\|_{L^p(\ell_2)} \leq C_d \frac{p^2}{p-1}\|f\|_{L^p(\T^d;S^p)}\label {LP},\\
&&\|f\|_{L^p(\T^d;S^p)} \leq C_d \frac{p^2}{p-1}\|(S_{E_j} f )_{j\geq0}\|_{L^p(\ell_2)}  \label{LP1},
\end{eqnarray}
for  all $f\in L^{p}(\T^d;S^p)$  and $1< p<\infty$.

\end{lemma}
\begin{proof} This lemma is well-known. We explain here that the dependence of the constants on $p$ is  in the order of   $\frac{p^2}{p-1}$. Given  $\varepsilon_j=\pm1$, let  $M_\varepsilon=\sum_{j\geq0} \varepsilon_j S_j$. Our choice of $S_{j}$'s makes $M_{\epsilon}$ a so-called H\"ormander-Mikhlin multiplier, which in particular is a Calderon-Zygmund operator. So it is bounded from the classical Hardy space $H^1$ to $L^1$. Moreover, it  is from $H^1$ to $H^1$ since it commutes with the classical Hilbert transform. By \cite[Theorem 6.4]{Me07}, it extends to a bounded operator on the semi-commutative BMO space BMO$_{cr}(L^{\infty}(\T^d)\bar{\otimes} \B(\ell^2(\Z^d)))$. (\ref{LP}) then follows from the interpolation result \cite[Theorem 6.2]{Me07} and the Khintchine inequality (\ref{khin}). 
(\ref{LP1}) follows from (\ref{LP}) by duality because of the identity $\langle f, g\rangle=\sum_j\langle S_{E_j}f, S_{E_j}g\rangle=\sum_j\langle S_jf, S_{E_j}g\rangle$.

\end{proof}

\begin{lemma}\label{Hilbert}
Suppose $R_j$ is a family of boxes with sides parallel to the axes in $\R^d$. Then there is a constant $C_d>0$ that depends only on d, such that for all $1<p<\infty$,  and for all families of measurable functions $f_j$ on $\R^d$, we have
\begin{equation}
	\left\| \left(\sum_j | S_{R_j}(f_j) |^2\right )^{\frac12}\right\|_{L^p(\T^d;S^p)} \leq C_d \left(\frac{p^2}{p-1}\right)^d \left\| \left( \sum_j |f_j|^2\right)^{\frac12}\right\|_{L^p(\T^d;S^p)}.
\end{equation}
\end{lemma}
\begin{proof} Assume $d=1$, $R_j=[a_j,b_j]$. Let $T_{a}$ to be the operator that sends $f(\cdot)$ to $f(\cdot)e^{i2\pi a_j(\cdot)}$. Let $P_+$ be the analytic projection. Then
$$S_{R_j}=T_{a_j}P_+T_{-a_j}-T_{b_j}P_+T_{-b_j}.$$
Note that $|T_{a_j}f_j|^2=|f_j|^2$ and we obtain the inequality for $d=1$ by the boundedness of $P_+$.
The case $d>1$ holds due to Fubini's theorem.
\end{proof}

	\begin{lemma}\label{lem:Cauchy-S}
		For sequences $\left(a_n \right), \left(c_n \right) \in B(H)$, we have
\begin{equation}\label{eq:C-S}
	\left|\sum_{n} a_n^* c_n\right|^2 \leq \left\|\sum_{n} a^*_na_n \right\| \left( \sum_{n} c^*_nc_n \right).
\end{equation}
	\end{lemma}
	\begin{proof} Given any $v\in S^2$, we have by Cauchy-Schwartz inequality that
		\[
		 tr \left( v^* \left|\sum_{n} a_n^* c_n  \right|^2 v \right)=  tr\left( \left|\sum_{n} a_n^* c_n v \right|^2 \right) \leq \left\|\sum_n a_n^*a_n \right\| tr \left[ v^* \left( \sum_{n}c^*_nc_n \right) v \right]   .
		\]
		Since $v$ is arbitrary, we obtain \eqref{eq:C-S}.
	\end{proof}

 \section{Proof of Theorem \ref{thm:main}}\label{sec:ProofofThm}

For $z\in \T^d$ given,  let $\Pi_z$ be  the *-homomorphism on  $\B(\ell^2(\Z^d))$ defined as  $$\Pi_z(A)=U_zAU_z^*$$ with $U_z$ the unitary sending  $e_{k}$ to $z^k e_{k}$.  
It is easy to see that $\Pi_z$ has the following presentation\begin{equation}\label{eq:transference}
	\Pi_z: A=(a_{k,j})\longmapsto (a_{k,j} z^{k-j}),
\end{equation}
with $k,j\in \Z^d$. $\Pi_z$ defines an isometric isomorphism  on  $\B(\ell^2(\Z^d))$  and     $S^p(\ell^2(\Z^d))$    for all $1\leq p<\infty$ because all these  norms are unitary invariant.
 Considering $z$ as a variable on $\T^d$, define $\Pi: \B(\ell^2(\Z^d))\to L^{\infty}(\T^d)\otimes \B(\ell^2(\Z^d))$ as
\begin{equation}\label{eq:transference1}
	\Pi(A)(z)=\Pi_z(A),
\end{equation}
Then $\Pi$ is an isometric isomorphism from $\B(\ell^2(\Z^d))$ to $L^{\infty}(\T^d)\otimes \B(\ell^2(\Z^d))$ and from $S^p(\ell^2(\Z^d))$ to $L^{p}(\T^d; S^p(\ell^2(\Z^d)))$  for all $1\leq p<\infty$.

Given a symbol $m=(m (i,j))_{i,j\in \Z^d}$ and $ A=(a_{i,j})_{i,j\in\Z^d}\in S^p(\ell^2(\Z^d ))$, set
\begin{align}\label{eq:muliplierNotation}
&M_{l}(n)=\sum_{s\in\Z^d}m(s,s-n)\, e_{s,s},\quad M_r(n)=\sum_{s\in \Z^d } m(s+n,s)\,e_{s,s},\\
&A(n)=\sum_{s\in\Z^d } a_{s,s-n} e_{s,s-n}=\sum_{s\in\Z^d } a_{s+n,s} e_{s+n,s}, \nonumber
\end{align}
for $n\in \Z^d$. Here $e_{s,t}$ denotes the operator on $\ell^2(\Z^d)$ sending $e_t$ to $e_s$.
Then $M_l(n), M_r(n)\in \B(\ell^2(\Z^d ))$ with norm bounded by $ C$ , $A(n)\in S^p(\ell^2(\Z^d ))$ for all $n\in \Z^d$, and $$\Pi(A)(z)=\sum_n A(n)z^n,\ \ M_m(A)=\sum_n M_l(n)A(n)=\sum_n A(n)M_r(n).$$
Here  $M_l(n)A(n)$ and $A(n)M_r(n)$ denote the  products of  operators in $\B(\ell^2(\Z^d)).$
Denote by $f=\Pi(A)$, i.e.
\begin{equation}\label{f}
f(z)=\sum_{n\in\Z^d} A(n)z^n.
\end{equation}
Denote by $\Pi(S^p)$  the image of $\Pi$, i.e. the subspace of $L_p(\T^d;S^p)$ consisting of all $f$ in the form of (\ref{f}).
Define the operator-valued Fourier multiplier $T_M$ on $\Pi(S^p)$ as
\begin{align} 
	T_{M}f(z)
	=\sum_{n\in\Z^d} M_l(n)A(n)z^n \label{eq:transf1}
\end{align}
Note that $M_l(n)A(n)=A(n) M_r(n)$ for all $n\in \Z^d$, we can represent $T_M$ as a  multiplier from the right  that
\begin{align} 
	T_{M}f(z)
	= \sum_{n\in\Z^d } A(n) M_r(n) z^n. \label{eq:transf2}
\end{align}
$T_M$ is defined so that the following identity holds $$T_{M}f=\Pi (M_m (A)).$$
Since $\Pi$ is a trace preserving $*$-homomorphism, we have
\begin{equation}\label{trans3}
	\|A\|_{S^p}=\|f\|_{L^p(\T^d;S^p)}, \ \ \|M_mA\|_{S^p}=\|T_{M}(f)\|_{L^p(\T^d;S^p)}.
\end{equation}
In order to prove $M_m$'s boundedness on $S^p$, we only need to prove that $T_{M}$ is  bounded  on $\Pi(S^p)$, i.e. the subspace of $L_p(\T^d;S^p)$ consisting all $f$ in the form of (\ref{f}). By Lemma \ref{thm:Littwood-P} and the transference relation \eqref{trans3}, it is sufficient to show the following inequality
\begin{equation}\label{eq:L-P1}
	\left\| \left( \sum_{j\geq0} |S_{E_j}(T_{M}f)|^2 \right)^{\frac12} \right\|_{L_p(\T^d;S^p)}\leq C_d (\frac{p^2}{p-1} )^d\left\| \left( \sum_{j\geq0} |S_jf|^2 \right)^{\frac12} \right\|_{L_p(\T^d;S^p)}
\end{equation}
and its adjoint form  
\begin{equation}\label{eq:L-P2}
	\left\| \left( \sum_{j\geq0} |(S_{E_j}(T_{M}f))^*|^2 \right)^{\frac12} \right\|_{L_p(\T^d;S^p)}\leq C_d (\frac{p^2}{p-1} )^d\left\| \left( \sum_{j\geq0} |(S_jf)^*|^2 \right)^{\frac12} \right\|_{L_p(\T^d;S^p)}
\end{equation}
for $p\geq 2$. By duality, we will obtain $M_m$'s boundedness on $S^p$ for $1<p<2$ as well. 

We will use \eqref{eq:transf1} as the presentation of $T_M$ to prove \eqref{eq:L-P1} and will use the presentation \eqref{eq:transf2} to prove  \eqref{eq:L-P2}.  
Note that $E_j$ is symmetric so $(S_{E_j}(T_{M}f))^*=S_{E_j}(T_{M}f)^*$ and \begin{eqnarray*}
(T_{M}f)^*= ( \sum_{n\in\Z^d } A(n) M_r(n) z^n)^*=  \sum_{n\in\Z^d } (M_r(-n))^*(A(-n))^*  z^n.
\end{eqnarray*}
So, both $S_{E_j}(T_Mf)$ and $(S_{E_j}(T_mf))^*$ have the multiplier symbols   on the left. This allows us to write the corresponding squares  in the forms with  $M_r$ or $M_l$ seating in the middle for both $S_{E_j}(T_Mf)$ and $(S_{E_j}(T_mf))^*$ and avoid the usual trouble caused by the non-commutativity of the operator products. After noting these facts, the argument for the case $d=1$ is rather standard, which we record below.

\medskip

{ \it Proof of Theorem \ref{thm:main}} 
We now set $d=1$. By the notations in \eqref{eq:muliplierNotation}, the conditions \eqref{Marc} and  \eqref{Marr}  are equivalent to 
\begin{equation}\label{eq:d=1}
	 \sup_{j\geq0 } \left\|\sum_{n\in E_j} |M_l(n+1)-M_l(n)| \right\|_{\infty}<C, \quad \sup_{j\geq 0} \left\|\sum_{n\in E_j} |M_r(n+1)-M_r(n)| \right\|_{\infty}<C,
\end{equation}
here $E_j$ is defined as in (\ref{eq:dyadic}).
Following the definition (\ref{SR}), we denote by   $S_{(a,b)}$  as,
\begin{equation}
	S_{(a,b)}g(z)= \sum_{a<n<b} \hat{g}(n)z^n
\end{equation}
for $g(z)=\sum_{n\in\Z} \hat{g}(n)z^n\in L^{p}(\T;S^p)$. 
We will deliberately extend the use of this notation and set $$S_{(a,b)}g=-S_{(b,a)}g$$  when $a>b$. 
For $j\in \N$, write  $E_{j,1}=(-2^j,-2^{j-1}] , E_{j,2}=[2^{j-1},2^j)$. Denote by $2^{(j)}_1=-2^{j},2^{(j)}_2=2^{j}$ and 
\begin{eqnarray}
 \Delta M_l(n)=
\begin{cases}
    M_l(n)-M_l(n-1)  \quad n<0\\
    M_l(n+1)-M_l(n)  \quad n>0.
\end{cases}
\end{eqnarray}


By applying summation by parts and the presentation of $T_M f$ in \eqref{eq:transf1} and \eqref{eq:transf2}, we obtain
\begin{align}\label{eq:summationByPart1}
	S_{E_j}(T_{M}f) =\sum_{i=1,2}S_{E_{j,i}}(T_{M}f) &=&\sum_{i=1,2} \left(M_l(2_i^{(j-1)}) (S_{E_{j,i}} f)+\sum_{n\in E_{j,i}} \Delta M_l(n) (S_{(n, 2^{(j)}_i)}f)\right)\\
	&=&\sum_{i=1,2} \left((S_{E_{j,i}} f)M_r(2^{(j-1)}_i) +\sum_{n\in E_{j,i}}  (S_{(n, 2^{(j)}_i)}f)\Delta M_r(n)\right) \label{eq:summationByPart2}
\end{align}
with a $\Delta M_r(n)$ defined similarly.
We will use the presentation (\ref{eq:summationByPart1}) to prove (\ref{eq:L-P1}) and will use \eqref{eq:summationByPart2}
to prove  (\ref{eq:L-P2}). The arguments are   similar. So	 we will  only give the argument for (\ref{eq:L-P1}).  We will ignore the term $j=0$ in (\ref{eq:L-P1}) because that $\|S_{E_0}(T_Mf)\|_{L^p(\T;S^p)}\leq C\|f\|_{L^p(\T;S^p)}$. 


Note  $\Delta M_l(n)$ is a diagonal operator, we can write  $\Delta M_l (n)=a^*_n b_n $ with $a_n,b_n$ diagonal operators and $|a_n|^2=|b_n|^2=|\Delta M_l(n)| $, 
then by Lemma \ref{lem:Cauchy-S}, we have, for $i=1,2$,
\begin{align}
	\left|\sum_{n\in E_{j,i}}\Delta M_{\ell} (n) S_{(n, 2^{(j)}_i)} f\right|^2 &\leq \left\| \sum_{n\in E_{j,i}} |\Delta M_l(n)| \right\|_{\infty} \left( \sum_{n\in E_{j,i}}|b_nS_{(n,2^{(j)}_i)}f|^2 \right) \nonumber \\
	&\leq C   \left( \sum_{n\in E_{j,i}}|b_nS_{(n,2^{(j)}_i)} f)|^2 \right)\\
	&= C   \left( \sum_{n\in E_{j,i}}|S_{(n,2^{(j)}_i)}(b_nS_jf)|^2 \right).\label{eq:sumP2}
\end{align}

\noindent Thus, 
\begin{eqnarray}\left\| \left( \sum_{j\in\N}\left|\sum_{n\in E_{j,i}}\Delta M_{\ell} (n) S_{(n,2^{(j)}_i)} f\right|^2 \right)^{\frac12} \right\|_{L^p(\T;S^p)}
 \leq C^{\frac{1}{2}} \left\| \left( \sum_{j\in\N}\sum_{n\in E_{j,i}}\left|  S_{(n,2^{(j)}_i)} \left(b_n S_j f\right)\right|^2\right)^{\frac12} \right\|_{L^p(\T;S^p)}
 \nonumber 
 \end{eqnarray}
 By Lemma \ref{Hilbert}, we get 
\begin{align}
&\left\| \left( \sum_{j\in\N}\left|\sum_{n\in E_{j,i}}\Delta M_{\ell} (n) S_{(n,2^{(j)}_i)} f\right|^2 \right)^{\frac12} \right\|_{L^p(\T;S^p)} \leq C_2\frac{p^2}{p-1} C^{\frac{1}{2}}  \left\| \left( \sum_{j\in\N}\sum_{n\in E_{j,i}}\left|  \left(b_n S_j f\right)\right|^2\right)^{\frac12} \right\|_{L^p(\T;S^p)} \nonumber \\
&=  C_2\frac{p^2}{p-1} C^{\frac{1}{2}}  \left\| \left( \sum_{j\in\N} (S_jf)^*  \left(\sum_{n\in E_{j,i}}  |b_n|^2\right) S_jf \right)^{\frac12} \right\|_{L^p(\T;S^p)}   \nonumber\\
& =  C_2\frac{p^2}{p-1} C^{\frac{1}{2}}  \left\| \left( \sum_{j\in\N} (S_jf)^*  \left(\sum_{n\in E_{j,i}}  |\Delta M_l(n)|\right) S_jf \right)^{\frac12} \right\|_{L^p(\T;S^p)}   \nonumber\\
& \leq C_2\frac{p^2}{p-1}C \left\|\left( \sum_{j\in\N} | S_j f|^2\right)^{\frac12}\right\|_{L^p(\T;S^p)} \label{eq:sumP3}
\end{align}
Hence, by \eqref{eq:summationByPart1}, \eqref{eq:sumP3}, and Lemma \ref{Hilbert}
\begin{eqnarray*}\label{eq:normSum}
	&&\left\|\left(\sum_{j\in\N} |S_{E_{j,i}} (T_{M} f)|^2 \right)^{\frac12}\right\|_{L^p(\T;S^p)}\\
	 &\leq&  
	\left\|  \left(\sum_{j\in\N} \left|  M_l(2^{(j-1)}_i)(S_{E_{j,i}} f)\right|^2\right)^\frac12\right\|_{L^p(\T;S^p)} +C\frac{p^2}{p-1}C_2 \left\|\left( \sum_{j\in\N} | S_j f|^2\right)^{\frac12}\right\|_{L^p(\T;S^p)} \nonumber\\
	 &\leq&\left\|C \left(\sum_{j\in\N} \left| (S_{E_{j,i}} f)\right|^2\right)^\frac12\right\|_{L^p(\T;S^p)} +C\frac{p^2}{p-1}C_2 \left\|\left( \sum_{j\in\N} | S_j f|^2\right)^{\frac12}\right\|_{L^p(\T;S^p)} \nonumber\\
	 &=&\left\|C \left(\sum_{j\in\N} \left| (S_{E_{j,i}}S_j f)\right|^2\right)^\frac12\right\|_{L^p(\T;S^p)} +C\frac{p^2}{p-1}C_2 \left\|\left( \sum_{j\in\N} | S_j f|^2\right)^{\frac12}\right\|_{L^p(\T;S^p)} \nonumber\\
	&\leq& C \frac{p^2}{p-1} \left\|\left(\sum_{j\in\N} |S_j  f)|^2 \right)^{\frac12}\right\|_{L^p(\T;S^p)} 
\end{eqnarray*}
for $i=1,2$.
Therefore we finish the proof of  (\ref{eq:L-P1}). The arguments for the adjoint version of (\ref{eq:L-P2}) are similar.
We then complete the proof  of Theorem \ref{thm:main}.

 \begin{corollary}\label{cor3.1} (\cite[Corollary 3.5]{CGPT22a})  The following Mikhlin conditions imply the boundedness of $M_m$ on $S^p$ for all $1<p<\infty$,
\begin{eqnarray}
|m(s, s+k)-m(s,  s+k+1)|\leq \frac C{|k|} \label{Mkc}\\
|m(s+k,s)-m(s+k+1,s)|\leq \frac C{|k|} \label{Mkr}
\end{eqnarray}
\end{corollary}
 
\begin{proof} It is clear that the Mikhlin conditions  \eqref{Mkc},\eqref{Mkr} imply the Marcinkiewicz type conditions \eqref{Marc},\eqref{Marr}.
\end{proof}

\section{The case $d>1$.}\label{sec:highDimension}
In this part, we generalize Theorem \ref{thm:main} to the $d$-dimensional case. Before we proceed to the main statement of the theorem, we need to borrow some notations from the calculus of finite differences.

\begin{definition} Let $\sigma:\Z^d\to \mathbb{C}$ and ${j}=(j_1,\cdots,j_d)\in\Z^d$. Let $\{e_j\}_{j=1}^d$ be standard basis of $\Z^d$, i.e., the $j$-th entry of $e_j$ is 1 and 0 for other entries, for $j=1,\cdots,d$. We define the forward partial difference operators $\Delta_{t_j}$by 
\begin{equation}\label{eq:differenceNotion}
	\Delta_{t_j} \sigma (t) \coloneqq \sigma(t+ e_j)-\sigma(t), \quad
\end{equation}
and for $\mathbf{\alpha}\in \{0,1\}^d$, define
\begin{equation*}
	\Delta_t^{\alpha} \coloneqq \Delta_{t_1}^{\alpha_1}\cdots \Delta_{t_d}^{\alpha_d}.
\end{equation*}
\end{definition}
For $\alpha=(1,\cdots 1)\in \{0,1\}^d$, we simplify the notation $\Delta_t^\alpha$ as $\Delta_t$.
Readers can find more information of calculus of finite differences in chapter 3 of \cite{Ruzhansky2010}. 
\subsection{The case $d=2$}
Recall that we have the   partition $\Z^2=  \bigcup_{j\geq 0} E_j 	$ with $E_j$ defined as \begin{equation}\label{eq:partitionZ2}
	E_j=\begin{cases}
		 \{(0,0) \}  & j=0,\\
		 \{(n_1,n_2)\in\Z^2:2^{j-1}\leq |(n_1,n_2)|_{\infty}<2^j \} & j\geq 1.
	\end{cases}
\end{equation}

\begin{theorem} \label{thm:dimension2}
Given $m=(m_{s,t})_{s,t\in\Z^2}\in \mathcal{B}(\ell^2(\Z^2))$. Suppose $m$ satisfies that
\begin{itemize}
	\item[i)] $\sup_{s,t\in\Z^2} |m_{s,t}|<C_1$
	\item[ii)] For any $k\in\N, s\in\Z^2 $,
	there are constants $C_2, C_3$ such that
	\begin{align}\label{eq:left}
	\left(\sum_{t=(t_1,\pm 2^{k-1})\in E_k} |\Delta_{t_1} m_{s,s+t}|  + \sum_{t=(\pm 2^{k-1}, t_2)\in E_k} |\Delta_{t_2} m_{s,s+t}| \right) &<C_2, \\
	 \sum_{ t=(t_1,t_2)\in E_k} |\Delta_t m_{s,s+t}|&<C_3, \label{eq:leftDouble}
	\end{align}
	and 
 \begin{align}\label{eq:right}
	\left(\sum_{t=(t_1,\pm 2^{k-1})\in E_k} |\Delta_{t_1} m_{s+t,s}|  + \sum_{t=(\pm 2^{k-1}, t_2)\in E_k} |\Delta_{t_2} m_{s+t,s}| \right) &<C_2\\
	 \sum_{ t=(t_1,t_2)\in E_k} |\Delta_t m_{s+t,s}|&<C_3, \label{eq:rightDouble}
	 \end{align}
\end{itemize}
Then $M_m$ is a  bounded Schur multiplier on $S^p(\ell^2(\Z^2))$ for $p\in (1,\infty)$ with an upper bound $\lesssim  (\frac{p^2}{p-1})^4$. Here $C_1, C_2$ and $C_3$ are positive absolute constants. 
\end{theorem}

 Now we come to the proof of Theorem \ref{thm:dimension2}. As explained at the beginning of Section 3, we only need to prove \eqref{eq:L-P1} and its adjoint version. Recall that $S_{E_j}$ is the partial sum projection on $L^p(\T^2;S^p(\Z^2))$ given by 
$
S_{E_j}f(z)=\sum_{n\in E_j} \widehat{f}(n) z^n
$, where $z\in \T^2$.


Applying the definition of $M_l$ \eqref{eq:muliplierNotation}, we see that  \eqref{eq:left} and \eqref{eq:leftDouble} imply 
\begin{align} \label{eq:leftPartialMl}
  \left\| \sum_{n=(n_1,\pm 2^{j-1})\in E_j} |\Delta_{n_1} M_{l}  (n)| \right\|_{\infty}  +\left\| \sum_{n=(\pm 2^{j-1},n_2)\in E_j}  |\Delta_{n_2} M_{l} (n )| \right\|_{\infty}  & < C_2,   \\
 \left\|\sum_{n=(n_1,n_2)\in E_j} \left|  \Delta_n M_{l}  (n) \right| \right\|_{\infty}& < C_3.   \label{eq:leftDoublePartial}
\end{align}

To prove \eqref{eq:L-P1}, we will cut $E_j$ into 4 rectangles $E_{j,k}, k=1,\cdots 4$ for $j\geq 1$. Let $I_{j}=[2^{j-1}, 2^j)\cap \Z, J_{j}=[-2^{j-1},2^j)\cap \Z$, set
\begin{equation*}
	E_{j,1}=J_j\times I_j,\qquad E_{j,2}=(-I_j)\times J_j, \qquad E_{j,3}= I_j\times (-J_{j}),\qquad E_{j,4}=(-J_j)\times (-I_{j}).
\end{equation*}

Thus, we have
\begin{equation}\label{eq:combin}
	S_{E_j}T_{M} f= \sum_{i=1}^4 S_{E_{j,i}}T_{M}f.
\end{equation}
To prove \eqref{eq:L-P1}, it is sufficient to prove
\begin{equation}\label{eq:sufficient2}
\left\| \left( \sum_{j = 0}^{\infty} \left| S_{E_{j,i }} T_{M} f\right|^{2} \right)^{\frac{1}{2}} \right\|_{L^{p}(\T^2;S^p)} \leq  C^{\prime} \left(\frac{p^2}{p-1}\right)^2 \left\| \left( \sum_{j = 0}^{\infty} \left| S_j  f  \right|^{2} \right)^{\frac{1}{2}} \right\|_{L^{p}(\T^2;S^p)}
\end{equation}
for $p\geq 2, i=1,2,3,4$. The arguments  for $i=1,2,3,4$ are similar. We will give the argument for $i=1$ only. By the fundamental theorem of calculus,
\begin{align}\label{eq:abelianSummationDimension2}
	&S_{E_{j,1}} T_{M} f \\
	= &M_l(-2^{j-1}, 2^{j-1}) S_{E_{j,1}}f + \sum_{n_1\in J_j} \Delta_{n_1} M_l(n_1, 2^{j-1}) S_{(n_1,2^j)\times I_j} f \nonumber \\
 &+\sum_{n_2\in E_j} \Delta_{n_2} M_l (-2^{j-1},n_2) S_{J_j\times (n_2,2^j)} f + \sum_{n=(n_1,n_2)\in E_{j,1}} \Delta_n M_l (n_1,n_2) S_{(n_1,2^j)\times (n_2,2^j)}f \nonumber\\
 &=: P^1_j + P^2_j + P^3_j + P^4_j
\end{align}
By the operator inequality  $|\sum_{k=1}^n a_k|^2\leq n \sum_{k=1}^n |a_k|^2$, we have
\begin{equation}\label{eq:Dimension2Part2}
	|S_{E_{j,1}} T_M f|^2 = |P^1_j + P^2_j + P^3_j + P^4_j|^2 \leq 4(|P^1_j|^2+ |P^2_j|^2+|P^3_j|^2+|P^4_j|^2).
\end{equation}
For part $P^1_j$, by assumption $i)$ of Theorem \ref{thm:dimension2}, we have
\begin{equation}\label{eq:partA}
	|P^1_j|^2= | M_l(-2^{j-1}, 2^{j-1}) ~S_{E_{j,1}}f|^2 \leq C_1^2 |S_{E_{j,1}}f|^2=C_1^2 |S_{E_{j,1}}S_jf|^2.
\end{equation}
By Lemma \ref{Hilbert}, 
\begin{align} 
\left\| \left(\sum_{j\geq 0}  |P^1_j|^2 \right)^{\frac12} \right\|_{L^p(\T^2;S^p(\Z^2))} \leq C  \left( \frac{p^2}{p-1} \right)^2\left\| \left(\sum_{j\geq 0}  |S_jf|^2 \right)^{\frac12} \right\|_{L^p(\T^2;S^p(\Z^2))}.
\end{align}
For part $P^2_j$, we follow  the  arguments similar to \eqref{eq:sumP2} and \eqref{eq:sumP3} in the one dimensional case and write $\Delta_{n_1} M_l(n_1, 2^{j-1})=a_n^*b_n$ with $|a_n|^2=|b_n|^2=|\Delta_{n_1} M_l(n_1, 2^{j-1})|$. Denote by $R_{n_1,j}=(n_1, 2^j)\times I_j$, we have  
\begin{align}\label{eq:partB1}
	|P^2_j|^2 &= \left|\sum_{n_1\in J_j} \Delta_{n_1} M_l(n_1, 2^{j-1}) S_{(n_1,2^j)\times I_j} f \right|^2  \nonumber\\
	& \leq \left\|\sum_{n_1\in J_j} |\Delta_{n_1}M_l(n_1,2^{j-1})| \right\|_{\infty} \left( \sum_{n_1\in J_j} \left|S_{R_{n_1,j}} (b_nS_{j}f) \right|^2 \right).
\end{align}

Thus, by \eqref{eq:partB1}, \eqref{eq:leftPartialMl} and Lemma \ref{Hilbert} and following  the  arguments similar to the case $d=1$, we get  
\begin{align}\label{eq:partB3}
\left\| \left(\sum_{j\geq 0}  |P^2_j|^2 \right)^{\frac12} \right\|_{L^p(\T^2; S^p(\Z^2))}
\leq  C_2 \left(\frac{p^2}{p-1}\right)^2 &\left\|  \left(  \sum_{j\geq 0}  |S_jf|^2  \right)^{\frac12}  \right\|_{L^p(\T^2;S^p)}
\end{align}
Similarly, we have 
\begin{equation} \label{eq:partC}
\left\| \left(\sum_{j\geq 0}  |P_j^3|^2 \right)^{\frac12} \right\|_{L^p(\T^2; S^p)}  \leq    C_2 \left(\frac{p^2}{p-1}\right)^2 \left\|  \left(  \sum_{j\geq 0}  |S_jf|^2  \right)^{\frac12}  \right\|_{L^p(\T^2;S^p)}.
\end{equation}
Now we come to the estimate of part $P^4_j$. Denote $R_{n,j}=(n_1,2^j)\times (n_2,2^j)$. Similarly,
\begin{align}\label{eq:partD}
\left\|  \left( \sum_{j\geq 0}  |P^4_j|^2 \right)^{\frac12} \right\|_{L^p(\T^2;S^p)} & \leq C_3^{\frac12}  \left\|  \left(  \sum_{j\geq 0} \sum_{n=(n_1,n_2)\in E_{j,1}} \left|  S_{R_{n,j}} S_j |\Delta_n M_l(n)|^{\frac12} f    \right|^2   \right)^{\frac12}  \right\|_{L^p(\T^2;S^p)}   \nonumber\\
\leq C_3^{\frac12} \left(\frac{p^2}{p-1} \right)^2  & \left\|  \left(  \sum_{j\geq 0} \sum_{n=(n_1,n_2)\in E_{j,1}} \left| S_j |\Delta_n M_l(n)|^{\frac12} f    \right|^2   \right)^{\frac12}  \right\|_{L^p(\T^2;S^p)} \nonumber \\
\leq C_3 \left(\frac{p^2}{p-1} \right)^2  &\left\| \left(\sum_{j\geq 0} |S_jf|^2  \right)^{\frac12}  \right\|_{L^p(\T^2;S^p)}.
\end{align}

Therefore, by \eqref{eq:combin}, \eqref{eq:abelianSummationDimension2} and \eqref{eq:partB1},\eqref{eq:partB3}, \eqref{eq:partC}, \eqref{eq:partD}, we have
\begin{align}
&\left\| \left( \sum_{j = 0}^{\infty} \left| S_{E_{j,1 }} T_{M} f\right|^{2} \right)^{\frac{1}{2}} \right\|_{L^{p}(\T^2;S^p)}\leq C^{\prime} \left(\frac{p^2}{p-1}\right)^2 \left\| \left( \sum_{j = 0}^{\infty} \left| S_{j } f\right|^{2} \right)^{\frac{1}{2}} \right\|_{L^{p}(\T^2;S^p)}
\end{align}
Thus, \eqref{eq:sufficient2} is proved. Hence, we finish the proof of  theorem \ref{thm:dimension2}.

\subsection{Higher dimensional case}
We need some additional notations to deal with the case $d>2$. Borrowing the notations from \cite{Hytoenen2016},  we denote by $\Z^{\alpha}$ the space,
\[
\Z^{\alpha}:=\{(n_i)_{i:\alpha_i=1}:n_i\in \Z \}
\]
 for  $\alpha\in \{0,1\}^d$. For any $n\in\Z^d$ and $E= I_1\times \cdots \times I_d\subseteq \Z^d$, let
\[
	n_{\alpha}\coloneqq (n_i)_{i:\alpha_i=1}\in \Z^{\alpha}, \quad E_{\alpha}\coloneqq \prod_{i:\alpha_i=1} I_i\subseteq \Z^{\alpha}
\]
be their natural projections onto $\Z^{\alpha}$. In particular, we will use the splittings $n=(n_{\alpha},n_{1-\alpha})\in \Z^{\alpha}\times \Z^{\mathbf{1}-\alpha}$ and $E=E_{\alpha}\times E_{\mathbf{1}-\alpha}$ where $\mathbf{1}=(1,\dots,1)$.  Suppose $s,t\in \Z^d$ and  we  abbreviate the interval notation  $[s,t)\cap \Z^d$ as $[s,t)$.



Similarly, denoted by  $\mathscr{I}^d $ the partition $\mathscr{I}^d \coloneqq \{E_j:j\geq 0\}$ of $\Z^d$, where 

\begin{equation}\label{eq:partitionZd}
	E_j=\begin{cases}
		 \{(0,\cdots,0) \}  & j=0,\\
		 \{(n_1,\cdots,n_d)\in\Z^d:2^{j-1}\leq |(n_1,\cdots,n_d)|_{\infty}<2^j \} & j\geq 1.
	\end{cases}
\end{equation}

Each $E_j$ can be further decomposed into $2^d(2^d-1)$ subsets and each of the subsets can be obtained by translation of the cube $F_j=[2^{j-1},2^j)\times \cdots \times[2^{j-1},2^j)$. 
Following similar procedures to those in the two dimensional case and using  the following discrete fundamental theorem formula,  
\begin{align}\label{eq:fundamentalCalculus}
	\chi_{[s,t)} (n) m(n) &= \chi_{[s,t)} \sum_{\alpha\in \{0,1\}^d} \sum_{k_\alpha\in [s,n)_{\alpha}} \Delta^{\alpha} m(s_{\mathbf{1}-\alpha},k_{\alpha}) \nonumber \\
	&= \sum_{\alpha\in \{0,1\}^d} \sum_{k_{\alpha}\in [s,t)_{\alpha}} \chi_{[k,t)_{\alpha}\times [s,t)_{\mathbf{1}-\alpha}}(n) \Delta^{\alpha}m(s_{\mathbf{1}-\alpha},k_{\alpha})
\end{align}
we can obtain the following theorem. The details are left to the interested readers.
\begin{theorem}\label{higerd}
	Given $m=(m_{s,t})_{s,t\in\Z^d}\in \B(\ell^2(\Z^d))$. Suppose $m$ satisfies that, for some $C>0$, 
	\begin{itemize}
	\item[i)] $\sup_{s,t\in\Z^d} |m_{s,t}|<C$,
		\item[ii)] For any $n\in\N$, $s\in \Z^d, \alpha\in \{0,1\}^d$,  $\alpha\neq 0$, and  any $r^{(n)}\in \Z^d$ satisfying $|r^{(n)}_i|=2^{n-1}$ for all $i=1,\cdots, d$,
		\begin{equation}\label{Marcrd>2}
		\sum_{t=(t_{\alpha}, r^{(n)}_{\mathbf{1}-\alpha})\in E_n  } |\Delta_t^{\alpha} m_{s,s+t}|<C,\quad \sum_{t=(t_{\alpha}, r^{(n)}_{\mathbf{1}-\alpha})\in E_n  } |\Delta_t^{\alpha} m_{s+t,s}|<C ,
		\end{equation}
		 	\end{itemize}
	Then $M_m$ extends to a  bounded Schur multiplier on $S^p(\ell^2(\Z^d))$ for $p\in (1,\infty)$ with an upper bound $C_d(\frac{p^2}{p-1})^{d+2}$. Here  $C_d$ is a constant  dependent only on the dimension $d$.
\end{theorem}




Note that we cannot hope for an analogue of Theorem \ref{higerd}  with $E_n$ defined by  the $\ell^2$-metric instead of  the $\ell^\infty$ metric  because the ball-type Fourier multipliers are not uniformly bounded on $L^p(\T^d)$ for any $d>1$.  Doust and Gillespie gave an example of ball-type Schur multipliers for the case $d=1$  in \cite[Theorem 6.2]{DoGi05}. Their argument does not seem to extend to the case $d>1$. 
\begin{example} ({\it Ball Schur Multipliers}) \label{example4.4}
Let $X_{0}= \{(0,0)\}\subset {\Bbb Z}^d\times \Z^d$. For $i\in\N$, let
 $$X_{i}= \{(k,j)\in {\Bbb Z}^d\times\Z^d; 2^{i-1}\leq |(k,j)|_2< 2^i \}.$$   Let $m_X=\sum_i \varepsilon_i \mathbbm{1}_{X_i}$ with $|\varepsilon_{i}|\leq 1$.  Then  $|\Delta_t^{\alpha} m|\leq 2^{|\alpha|}$. Note  that
 $$ |(k,j)|_2\simeq |k|_2+|j|_2\simeq |k-j|_2+|j|_2\simeq |k-j|_\infty+|j|_\infty \simeq |k-j|_\infty+|k|_\infty.$$ We can find a constant $K_d$ which only depends on $d$ such that the set $$\{(s_0,s_0+t); t\in E_n\}\bigcup \{(s_0+t,s_0); t\in E_n\}$$ intersects with at most $K_d$ many $X_{i}'s$ for any fixed $s_0$. Since $\bigcup_{0\leq i\leq n} X_i$ is convex for all $n$, we  conclude that  there are at most $2^dK_d$ many non-zero terms in the two summations in \eqref{Marcrd>2}, and the summations are bounded by $2^dK_d$. So \eqref{Marcrd>2} is satisfied and $M_m=\sum_{i} \varepsilon_{i } P_{X_i} $  is bounded on $S^p$ for any $1<p<\infty$.

\end{example}

\subsection{The case of continuous indices}
We explain in this section that Theorem \ref{thm:main} and Theorem \ref{higerd} extend to the continuous case by approximation. 
Let $S^p(\R^d)$ be the  space of Schatten $p$-class operators acting on the Hilbert space $L^2(\R^d)$.
We identify $ S^2(\R^d)$ as $L^2(\R^{d}\times \R^{d} )$, so for $A\in S^2(\R^d)$ we can talk about its pointwise value $a_{s,t}$. For $m\in L^\infty(\R^{d}\times\R^{d})$, we consider the Schur multiplier type map 
$$M_m(A)=(m(s,t)a_{s,t})_{s,t\in \R^d}.$$
Motivated by the work of \cite{LaSa11} and \cite{CGPT22a}, we wish to find sufficient conditions on $m$ so that $M_m$ extends to a bounded map with respect to the $S^p$-norm for $1<p<\infty$.



\begin{theorem}
For $p\in (1,\infty)$, consider the Schur multiplier $M_m$ on $S^p(\R)$ with symbol $m (\cdot,\cdot)$ in  $L^{\infty} (\mathbb{R}^{2})$ whose partial derivatives are continuous on $(-2^{j+1},-2^j)\bigcup (2^j, 2^{j+1})$ for all $j~\in~\Z$. Suppose there exists an absolute constant $C$ such that, for all  $j \in \mathbb{Z}$ and $x,y \in \mathbb{R}$, 

\begin{equation} \label{eq: con 1}
    \int_{- 2^{j + 1}}^{- 2^{j}} \left| \partial_{1}  m   (y + t, y)  \right| \ d t + \int_{2^{j}}^{2^{j + 1}} \left|  \partial_{1}  m  (y + t, y) \right| \ d t  \leq C
\end{equation}
and
\begin{equation} \label{eq: con 2}
  \int_{- 2^{j + 1}}^{- 2^{j}} \left| \partial_{2}  m    (x, x + t) \right| \ d t + \int_{2^{j}}^{2^{j + 1}} \left|  \partial_{2} m  (x, x + t) \right| \ d t   \leq C.
\end{equation}
Then, the Schur multiplier $M_{m}$ extends to a   bounded map on $S^p(\R)$ with $\| M_{m} \|\leq C \max \left\{ p^{3}, \frac{1}{(p - 1)^{3}} \right\}$.
\end{theorem}
\begin{proof}  Let
 $\mathcal{D}_{k} $ be the $\sigma$-algebra generated by  dyadic cubes $Q_{k, s, t}= \left(\frac {s}{2^k},\frac {s+1}{2^k} \right] \times \left( \frac {t}{2^k},\frac {t+1}{2^k} \right],s,t\in \Z$.  Then  $(\mathcal{D}_{k})_{k = 1}^{\infty}$ is the usual dyadic   filtration for $\mathbb{R}^{2}$.
Given $m\in L^\infty(\mathbb{R}^{2})$, let $m_k=\mathbb{E}_k (m)$   the conditional expectation of $m$ with respect to the $\sigma$-algebra $\mathcal{D}_{k}$. That is to say 
\[
m_{k} (x) = \sum_{Q \in \mathscr{D}_{k}}  \frac{1}{|Q|} \left[ \int_{Q} m(y) dy \right] \chi_{Q} (x), \quad \forall x\in \R^2.
\] 
  Let $L^2(\R, \mathcal{D}_{k})$ be the $L^2$ space of all $\mathcal{D}_k$-measurable functions. 
  Let $\widetilde{m_{k}}(s,t) = m \left( \frac {s}{2^k},\frac t{2^k} \right)$ for ${s, t \in \mathbb{Z}}$. Note that $S^p(L^2(\R, \mathcal{D}_{k}))$ is isometrically isomorphism to $  S^p(\ell^2(\Z))$. We see that $M_{{\tilde m}_k}$ extends to a bounded Schur multiplier on $S^p(\ell^2(\Z))$ with the same norm if 
 $M_{m_k}$ extends to a bounded Schur multiplier on $S^p(L^2(\R, \mathcal{D}_{k}))$ and vice versa.
By Lemma 1.11 of \cite{LaSa11},
\[
\left\| M_m \right\|={ \widebar\lim}_{k \to \infty} \left\| M_{m_k} \right\| ={ \widebar\lim}_{k \to \infty} \left\| M_{\widetilde{m}_k} \right\| .
\]

So, we need to show that $M_{\widetilde{m_{k}}}$ satisfies conditions (\ref{Marc}), (\ref{Marr}). First, we verify condition (\ref{Marr}). For each $j \in \mathbb{N}$ and $s \in \mathbb{Z}$,

\begin{align} \label{eq:continue}
&\sum_{\ell = 0}^{2^{j} - 2} \left| m_{k} \left( s, s + 2^{j} + \ell + 1 \right) - m_{k} \left( s, s + 2^{j} + \ell \right) \right|  \nonumber\\
&= 2^{2 k}\sum_{\ell = 0}^{2^{j} - 2} \left| \int_{\frac{s}{2^{k}}}^{\frac{s + 1}{2^{k}}} \int_{\frac{2^{j} + \ell}{2^{k}}}^{{\frac{2^{j} + \ell + 1}{2^{k}}}} M ( y, y + x + \frac{1}{2^{k}} ) \ d x \ d y - \int_{\frac{s}{2^{k}}}^{\frac{s + 1}{2^{k}}} \int_{\frac{2^{j} + \ell}{2^{k}}}^{{\frac{2^{j} + \ell + 1}{2^{k}}}} M \left( y, y + x \right) \ d x \ d y \right|  \nonumber\\
&= 2^{2 k}\sum_{\ell = 0}^{2^{j} - 2} \left| \int_{\frac{s}{2^{k}}}^{\frac{s + 1}{2^{k}}} \int_{\frac{2^{j} + \ell}{2^{k}}}^{{\frac{2^{j} + \ell + 1}{2^{k}}}} \int_{y + x}^{y + x + \frac{1}{2^{k}}}  \partial_{2} M \left( y, t \right) \ d t \ d x \ d y \right|  \nonumber\\
&\leq 2^{2 k}\sum_{\ell = 0}^{2^{j} - 2} \int_{\frac{s}{2^{k}}}^{\frac{s + 1}{2^{k}}} \int_{\frac{2^{j} + \ell}{2^{k}}}^{{\frac{2^{j} + \ell + 1}{2^{k}}}} \int_{y + x}^{y + x + \frac{1}{2^{k}}} \left| \partial_{2} M \left( y, t \right) \right| d t \ d x \ d y  \nonumber\\
&= 2^{2 k} \int_{\frac{s}{2^{k}}}^{\frac{s + 1}{2^{k}}} \int_{\frac{2^{j}}{2^{k}}}^{{\frac{2^{j + 1} - 1}{2^{k}}}} \int_{y + x}^{y + x + \frac{1}{2^{k}}} \left| \left[ \partial_{2} (M) \right] \left( y, t \right) \right| d t \ d x \ d y  \nonumber\\
&= 2^{2 k} \int_{\frac{s}{2^{k}}}^{\frac{s + 1}{2^{k}}} \int_{\frac{2^{j}}{2^{k}}}^{{\frac{2^{j + 1} - 1}{2^{k}}}} \int_{y + \frac{2^{j}}{2^{k}}}^{{y + \frac{2^{j + 1}}{2^{k}}}} \chi_{\left( x, x + \frac{1}{2^{k}} \right)} (t) \left| \left[ \partial_{2} (M) \right] \left( y, t \right) \right| d t \ d x \ d y  \nonumber\\
&= 2^{2 k} \int_{\frac{s}{2^{k}}}^{\frac{s + 1}{2^{k}}} \int_{y + \frac{2^{j}}{2^{k}}}^{{y + \frac{2^{j + 1}}{2^{k}}}} \int_{\frac{2^{j}}{2^{k}}}^{{\frac{2^{j + 1} - 1}{2^{k}}}} \chi_{\left( t - \frac{1}{2^{k}}, t \right)} (x) \ d x \left| \left[ \partial_{2} (M) \right] \left( y, t \right) \right|  d t \ d y  \nonumber\\
&\leq 2^{k} \int_{\frac{s}{2^{k}}}^{\frac{s + 1}{2^{k}}} \int_{y + \frac{2^{j}}{2^{k}}}^{{y + \frac{2^{j + 1}}{2^{k}}}} \left|  \partial_{2} M \left( y, t \right) \right|  d t \ d y 
\leq \sup_{y \in \mathbb{R}} \int_{\frac{2^{j}}{2^{k}}}^{{\frac{2^{j + 1}}{2^{k}}}} \left| \left[ \partial_{2} (M) \right] \left( y, y + t \right) \right|  d t \leq A.
\end{align}

The last inequality follows from the assumption in (\ref{eq: con 2}). So, condition (\ref{Marr}) is verified. Applying the same argument, we utilize the assumption in  (\ref{eq: con 1}) to prove condition (\ref{Marc}). Therefore,  by  Theorem \ref{thm:main}, $\left\| M_m \right\|={ \widebar\lim}_{k \to \infty}\left\| M_{\widetilde{m}_k} \right\| \leq C \max \left\{ p^{3}, \frac{1}{(p - 1)^{3}} \right\}$.
\end{proof}

\noindent Similarly,  Theorem \ref{higerd} and Example \ref{example4.4} have analogues  in the continuous case as well.
\begin{theorem}
Denote by $E_{j} := \{ t \in \mathbb{R}^{d} : 2^{j-1} \leq | t |_{\infty} < 2^{j } \}$ for $j\in\Z$.  For $p\in (1,\infty)$, consider the Schur multiplier $m \in L^{\infty} (\mathbb{R}^{2d})$ whose partial derivatives are continuous up to the boundary of $E_k$ for all $k\in\Z$. Assume there exists a constant $C$ such that, 
\begin{align} \label{eq:rightCondition}
&  \int_{(t_{\alpha}, r^{(j)}_{1-\alpha}) \in E_{j}} \left|  \partial^{\alpha} m ( s,s  + t ) \right| dt_\alpha  \leq C\\
	&   \int_{(t_{\alpha}, r^{(j)}_{1-\alpha}) \in E_{j}} \left|  \partial^{\alpha} m ( s+t, s   ) \right| dt_{\alpha}\leq C,
\end{align}
for any  $j\in \Z, s\in  {\Bbb R}^d$ and any $r^{(j)}\in \R^d$ with $|r^{(j)}_i|=2^{j-1}$ for all $1\leq i\leq d$.
Then, the Schur multiplier $M_{m}$ extends to a bounded operator on $S^p(\R^d)$ for all $1<p<\infty$ with $\| M_{m} \|\leq C_d \max \left\{ p^{d+2}, \frac{1}{(p - 1)^{d+2}} \right\}$. Here $t_\alpha$ is defined as in Section 4.2.

\end{theorem}

\begin{remark}The Schur multipliers in all theorems of this article are also   completely bounded  on $S^p$ for $1<p<\infty$, the arguments  are exactly the same.
\end{remark}

\section{ Discussions}
\subsection{Counterexamples}\label{sec:Counterex}
\begin{enumerate}
\item  We show in the following that (\ref{Marc}) alone is not sufficient for the boundedness of $M_m$. 

 Choose a large $K\in \N.$ 
Let  $m(s,t)=exp(\frac{i2\pi kj}K)$ if $s=2^k,t=2^j$ for some $j,k\in \N$ satisfying $1\leq j <k  \leq K,$ and $m(s,t)=0$ for other $s,t\in \N$. Let  $\tilde m(k,j)=exp(\frac{i2\pi kj}K)$ if $ 1\leq j <k  \leq K,$ and $\tilde m(k,j)=0$ for other $k,j\in \N$.
Let $U$ be the partial isometry on $\ell_2(\N)$ sending $e_{k}$ to $e_{2^k}$. 
Then, we have $M_{\tilde m}(A) = U^*M_m(UAU^*)U$ for any   $A\in S^p(\ell_2(\N))$ and $\|M_{\tilde m}\|\leq \|M_m\|$. 

Note that,  for any $N, j$ given, there exists at most one $k$ (actually $k=N$)  satisfying $k>j$  and $$2^{N-1}-1\leq|2^k-2^j|<2^N.$$  Using the fact that $|m(s,t)|\leq1$, we get, 
  and for any $N,t$ given
  \begin{eqnarray*}
  \sum_{2^{N-1}\leq |r|< 2^{N}}|m(t+r+1,t)-m(t+r,t)| \leq 2,
  \end{eqnarray*}
because there are at most two non-zero terms in the sum above. 
This means $(m(s,t))_{s,t}$ satisfies the row condition (\ref{Marc}).
On the other hand, if $s=2^N$, then $|m(s,t)-m(s,t+1)|$ does not vanish if $t$ or $t+1$ has the form of $2^j, j=1,\cdots, N-1$ by the definition of $m(s,t)$. Hence  we have
$$
\sum_{2^{N-1}\leq |t-s|<2^N}|m(s,t+1)-m(s,t)|=\sum_{2^{N-1}\leq s-t <2^N} |m(s,t+1)-m(s,t)|=2(N-1),
$$
which shows that $(m(s,t))_{s,t}$ fails the column condition \eqref{Marr}.

 Let $A$ be the  $K$ by $K$ matrix 
$(exp(\frac{-i2\pi kj }K))_{1\leq  k,j\leq K}$. Then $A$  has $S^p$ norm   ${K^{\frac12+\frac 1p}}$.  $M_{\tilde m}(A)$ is the  lower triangle matrix  with all nonzero coefficients being $1$ which has $S^p$ norm 
$\simeq K$ for any given $p, 1<p<\infty$.  
This shows that $K^{\frac12-\frac1p}\lesssim \|M_{\tilde m}\|\leq \|M_m\|$. We then conclude that (\ref{Marc}) alone is not sufficient for the boundedness of $M_m$.
 By the symmetricity, (\ref{Marr}) alone is not sufficient for the boundedness of $M_m$ either.

 \item A smooth version of the example above implies that neither  the assumption (\ref{Mkc}) nor the assumption (\ref{Mkr}) is removable in Corollary \ref{cor3.1}. Indeed, fix a large $K>0$, 
let $m_1(s,t)=\exp(\frac{i2\pi \log_2s\log_2 t }K)$ for $1\leq s\leq t\leq 2^K, s,t\in \N$;  $m_1(s,t)=\frac  {2^{K+1}-t}{ 2^{K}}\exp(i2\pi \log_2s)$ for $1\leq s\leq t, 2^K<t\leq 2^{K+1}, s,t\in \N$ and 
$m_1(s,t)=0$ otherwise.  Then 
$m_1$ satisfies $(\ref{Mkc})$ because 
$$\left| \frac{\partial} {\partial t}\exp \left( \frac{i2\pi \log_2s\log_2 t }K \right) \right|\lesssim \frac 1t\leq \frac 1{t-s}$$ whenever $s<t\leq 2^K$.  Assuming  the sufficiency of (\ref{Mkc}) would imply the uniform  boundedness of $M_{m_1}$ for all $1<p<\infty$, which is wrong because
$M_m(A)=M_{m_1}(VAV)$ for $A\in S^p(\ell^2(\N))$ and $m, V$ defined above. We conclude that  neither  the assumption (\ref{Mkc}) nor the assumption (\ref{Mkr}) is removable.

\item  Let $$F_{N,t}= \{(s,t)\in {\Bbb N}\times\N; 2^{N-1}\leq |s-t|< 2^{N} \}$$ for $N,t\in \N$. Let $Q_{N,t}$ be the projection from $S^2(\ell_2(\N))$ onto the span of $\{e_{s,t}; (s,t)\in F_{N,t}\}$.  One may wonder whether Corollary \ref{cor:main} can be improved so that the Schur multiplier $S_{\varepsilon}=\sum_{N,t\in\N}\varepsilon(N,t)Q_{N,t}$ is bounded for any sequence $|\varepsilon(N,t)|\leq 1$. This is   impossible as well.\footnote{We can   get the same conclusion for sequences $\varepsilon_k=\pm1$ by choosing  Hadamard   orthogonal matrices   instead of the  matrices $(exp(\frac{-i2\pi kj }K))_{1\leq  k,j\leq K}$.} To see this, let $\varepsilon(N,t)=exp(\frac{i2\pi Nj}K)$ if  $t=2^j$ for some $j\in \N$ and $j<N\leq K$. Let $\varepsilon(N,t)=0$ otherwise. Let $V$ be the projection on $\ell^2(\N)$ such that $V(e_i)=e_i$ if $i=2^k$ for some $k\in \N$,  and $V(e_i)=0$ otherwise. Then, for $M_m$ defined in the first example and $A\in S^p(\ell_2(\N))$, we have $$  M_m(A)=S_{\varepsilon}(VAV).$$ 
  Therefore, $K^{\frac12-\frac1p}\lesssim \|S_{\varepsilon}\|$.

  \end{enumerate}
\subsection{Operator Valued Symbol}\label{sec:OperatorValued}

The Schatten $p$ class has a natural operator space structure  inherited from the operator space complex interpolation $S^p=(S^\infty, S^1)_{\frac1p}, 1<p<\infty$. Pisier proved in \cite[Lemma1.7]{Pi98} that, with respect to $S^p$'s natural operator space structure, a map $M$ on $S^p(\ell^2)$ is completely bounded if and only if $M\otimes id_{ S^p(H)}$ is bounded on $S^p(S^p)=S^p(\ell^2\otimes H)$ for any separable Hilbert space $H$. We will  explain an operator-valued version of Theorem \ref{thm:main} which particularly implies   the complete boundedness of the Schur multipliers considered in Theorem \ref{thm:main}. We will assume the readers are familiar with the terminology of operator spaces in this subsection.

 We will consider $A\in S^2(\ell^2(\Z)\otimes  H)$ with $H$ a separable Hilbert space. We present $A$ in its matrix form $ (a_{i,j})_{i,j\in \Z}$ with $a_{i,j}\in S^2(H)$. More precisely, denote by $e_i$  the canonical basis of $\ell_2$, let $e_{j,i}$ be the rank one operator on $\ell^2$ sending $e_i$ to $e_j$. Denote by $tr$ (resp.  $\tau$)  the canonical trace on $B(\ell^2(\Z))$ (resp. $B(H)$). We set 
$$a_{i,j}=(tr\otimes id) (A (e_{j,i}\otimes id_H)).$$

Let $\M$ be a finite von Neumann algebra with a normal  faithful tracial state  $\tau$. Given a $\M$-valued bounded function $m$ on $\Z\times\Z$ and   $A\in S^2(\ell^2(\Z)\otimes  H)$ in its matrix form $(a_{i,j})_{i,j\in \Z}$. We define $M_m(A)$ as the matrix form 
\begin{eqnarray}
M_m(A)=(m_{i,j}\otimes a_{i,j})_{i,j}.\label{Mmoperator}
\end{eqnarray}
We will show that an analogue of Theorem \ref{thm:main} holds  that  there exists $C_p\simeq (\frac{p^2}{p-1})^3$ for $1<p<\infty$, such that
$$\|M_m(A)\|_{L^p(\M\otimes B(\ell^2(\Z)\otimes H))}\leq C_p\|A\|_ {S^p(\ell^2\otimes H)},$$ 
for all  $A\in S^2\cap S^p$. By the density of $S^2\cap S^p$,  $M_m$ extends to a  bounded operator from  $S^p(\ell^2(\Z) \otimes H)$ to $L^p(\M\otimes B(\ell^2(\Z)\otimes H))$ when $m$ satisfies Marcinkiewicz type conditions. When $\M=\Bbb{C}$, this implies the completely boundedness of $M_m$ in Theorem \ref{thm:main} by Pisier's result.
 We will need  Pisier's $L_\infty(\ell_1)$ norm to express this Marcinkiewicz type condition.

\begin{definition}({\it Pisier's $L^{\infty}(\ell_{1})$ norm}) 
Given $N$-tuples $(x_{1},\dots,x_{N})$ in $ \mathcal{M} $, set 
\begin{equation}\label{eq:factorization2}
\|x\|_{L^{\infty}(\mathcal{M};\ell_{1})}   = \inf\left\{ \left\|\left( \sum a_{j}a_{j}^{*} \right)^{\frac{1}{2}}\right\| \cdot \left\| \left(\sum b_{j}^{*}b_{j}\right)^{\frac{1}{2}}\right\| \right\},
\end{equation}
where the infimum runs over all possible factorizations $x_{j}=a_{j}b_{j}$ with $a_{j}, b_{j}\in  \mathcal{M}$.
\end{definition}
When $x_k\geq 0$,  $\|x\|_{L^{\infty}(\mathcal{M};\ell_{1})}=\|\sum_k |x_k|\|$ but the two quantities are not comparable in general. Pisier showed that $\|x\|_{L^{\infty}(\mathcal{M};\ell_{1})}<\infty$ if and only if there is a decomposition    $x_k=x_{k,1}-x_{k,2}+ix_{k,3}-ix_{k,4}$ such that $x_{k,\ell} \geq0$ and $\|(x_{k,\ell})_k\|_{L^{\infty}(\mathcal{M};\ell_{1})}<\infty$ for all $\ell=1,2,3,4$. 

Given $M_m$ defined as in \eqref{Mmoperator},  let $$\Delta_s m(s,t)=m(s+1,t)-m(s,t ),\ \  \Delta _t m(s,t)=m(s,t+1)-m(s,t),$$ for $s,t\in \Z$. 
 \begin{theorem} \label{thm:operator}
\vskip.1cm
 $M_m$ defined as in \eqref{Mmoperator} extends to a   bounded map from  Schatten $p$-classes $S^p(\ell^2\otimes H)$ to ${L^p(\M\otimes B(\ell^2(\Z)\otimes H))}$ for all $1<p<\infty$ with bounds $\lesssim (\frac{p^2}{p-1})^3$, if  $m$ is bounded  in $\M$ and there is a constant $C$ such that, 
 (i)  for any $n\in{\Bbb N}, t\in \Z$, 
 \begin{eqnarray}\label{deltam}
  \|(\Delta_s m(s+t,t))_{ 2^{n-1}\leq |s|<2^{n}}\|_{L^{\infty}(\mathcal{M};\ell_{1})}<C, 
\end{eqnarray}
(ii) for any $n\in{\Bbb N}, s\in \Z$,  \begin{eqnarray}\label{deltam2}
   \|(\Delta_t m(s,s+t))_{ 2^{n-1}\leq |t|<2^{n}}\|_{L^{\infty}(\mathcal{M};\ell_{1})}<C.
\end{eqnarray}  
 \end{theorem}
 
 \noindent {\it Sketch of proof.}  Denote by $\tilde m(s,t)=m(s,t)\otimes 1_{\ell_2  \otimes H}$ and $\tilde a_{s,t}=1_\M\otimes a_{s,t}$. Then $\tilde m(s,t)$ commutes with $a_{s',t'}$ for any $s,t,s',t'\in \Z$. Let 
 \begin{align}\label{eq:muliplierNotationO}
&\tilde M_{l}(j)=\sum_{s\in\Z}\tilde m(s,s-j)\otimes e_{s,s},\quad \tilde M_r(j)=\sum_{s\in \Z } \tilde m(s+j,s)\otimes e_{s,s},\\
&\tilde A(j)=\sum_{s\in\Z } \tilde a_{s,s-j}\otimes e_{s,s-j}=\sum_{s\in\Z } \tilde a_{s+j,s} \otimes e_{s+j,s}, \nonumber
\end{align} 
with $e_{s,t}$ the canonical basis of $S^2(\ell_2(\Z))$.
Let $f(z)
	= \sum_{j\in\Z }  \tilde A(j) z^j$ and 
\begin{align} 
	T_{\tilde M}f(z)
	= \sum_{j\in\Z }\tilde M_l(j)  \tilde A(j) z^j.
\end{align}
We still have  that
\begin{align} 
	T_{\tilde M}f(z)
	=\sum_{j\in\Z } \tilde A(j) \tilde M_r(j) z^j,
\end{align}
and the identities
\begin{eqnarray*}
\|f\|_{L^p(\M\otimes B(\ell^2(\Z)\otimes H))}=\|A\|_{S^p(\ell_2\otimes H)} \\
|T_{\tilde M}f\|_{L^p(\M\otimes B(\ell^2(\Z)\otimes H))}=\|M_m(A)\|_{S^p(\ell_2\otimes H)}.
\end{eqnarray*}
Moreover, the conditions  \eqref{deltam} and \eqref{deltam2} implies that \begin{eqnarray*}
 \|\Delta_l\tilde M(j)_{2^{n-1}<|j|\leq 2^n}\|_{  L^\infty(\M\otimes B(\ell^2(\Z)),\ell_1)}, \|\Delta_r\tilde M(j)_{2^{n-1}<|j|\leq 2^n}\|_{  L^\infty(M\otimes B(\ell^2(\Z)), \ell_1)}<C
\end{eqnarray*}
 for  $$\Delta_l \tilde M(j)=\tilde M_l(j+1)-\tilde M_l(j), \Delta_r \tilde M(j)=\tilde M_r(j+1)-\tilde M_r(j).$$  After these, it is not hard to check that the arguments for the proof of Theorem \ref{thm:main} work as  well for the tensor case.

\begin{corollary} The Schur multipliers  considered in Theorem \ref{thm:main} are completely bounded on the Schatten classes $S^p, 1<p<\infty$ with bounds $\lesssim (\frac{p^2}{p-1})^3$ with respect to their natural operator space structure.
\end{corollary}

\begin{remark} The optimal constants for the $L^p$ bounds of the classical Marcinkiewicz Fourier multipliers is $p^\frac32$ as $p\rightarrow \infty$ (\cite{TaWr01}). It is unclear at all what is the optimal  asymptotic order for the $S^p$-bounds of the Schur multipliers in Theorem \ref{thm:main}. 
\end{remark}

\noindent{ \bf Open Question.}  
  Assume    $m$ is a  bounded map on ${\Bbb Z}\times{\Bbb Z}$  such that
 \begin{eqnarray*}
 \sum_s |m(k,j_s)-m(k,j_{s+1})|^2&<&C
 \end{eqnarray*}
 for all possible increasing sequences $j_s\in {\Bbb Z}$. Does $M_m$ extend to a bounded map on  $S^p$ for all $1<p<\infty$?
 
\begin{remark} The authors heard this question from  Potapov and Sukochev. They told the authors that it stems from  the work of  Birman and Solomyak on double operator integrals. 
The third author  noticed Theorem \ref{thm:main} during his effort of attacking this question.

\end{remark}

\noindent {\bf Acknowledgments.} The authors are grateful to the referee for a very careful reading of the article. The second and third authors are partially supported by NSF grants  DMS 2247123.  

\medskip\medskip

\noindent {\bf Statements and Declarations}

\noindent On behalf of all authors, the corresponding author states that there is no conflict of interest.
\medskip

\noindent Data sharing not applicable to this article as no datasets were generated or analysed during the current study.

\begin{thebibliography}{1}


\bibitem {AlPe02} A. B. Aleksandrov, V. V. Peller, Hankel and Toeplitz-Schur multipliers, Math. Ann. 324 (2002), 277-327. 

\bibitem {Ar82} J. Arazy, Certain Schur-Hadamard multipliers in the spaces $C_p$, Proc. Amer. Math. Soc. 86 (1982),
59-64.

\bibitem {Be77} G. Bennett, Schur multipliers, Duke Math J. 44 (1977), 603-639.

\bibitem {BeGi94} B. Berkson and T.A. Gillespie, Spectral decompositions and harmonic analysis on UMD
spaces, Studia Math. 112 (1994), 13-49.


\bibitem  {Bou86} J. Bourgain, Vector-valued singular integrals and the $H^1$-$BMO$ duality, in: J.A. Chao, W. Woyczyński (Eds.), Probability Theory and Harmonic Analysis, Marcel Dekker, New York, 1986, 1-19.

\bibitem {BoFe84}
M.~Bo\.{z}ejko and G.~Fendler,
\newblock Herz-{S}chur multipliers and completely bounded multipliers of the
  {F}ourier algebra of a locally compact group.
\newblock {\em Unione Matematica Italiana. Bollettino. A. Serie VI},
  3(1984),297--302.

\bibitem {CPSZ19} M. Caspers, D. Potapov, F. Sukochev, D. Zanin, Weak type commutator and Lipschitz estimates: resolution of the Nazarov-Peller conjecture, Amer. J. Math. 141 (2019), 593-610.

\bibitem {CaSa15}
M.~Caspers and M.~de~la Salle,
\newblock Schur and {F}ourier multipliers of an amenable group acting on
  non-commutative {$L^p$}-spaces.
\newblock {\em Trans. Amer. Math. Soc.}, 367, no(10),6997--7013, (2015).


\bibitem {ClPaSuWi00}P. Clement, B. de Pagter, F.A. Sukochev and H. Witvliet, Schauder decompositions and
multiplier theorems. Studia Math. 138 (2000), 135-163.

\bibitem {CGPT22a}
J.~M. Conde-Alonso, A.~M. Gonz\'alez-P\'erez, J.~Parcet, and E.~Tablate.
\newblock Schur multipliers in {S}chatten-von {N}eumann classes. {\it Annals of Math.}, 198(2023), 1229-1260.
\href{https://arxiv.org/pdf/2201.05511v2}{arXiv:2201.05511v2} (2022), 1-20.

\bibitem {CGPT22b}
J.~M. Conde-Alonso, A.~M. Gonz\'alez-P\'erez, J.~Parcet, and E.~Tablate,
H\"ormander-Mikhlin criteria on Lie group von Neumann algebras,  \href{https://arxiv.org/pdf/arXiv:2201.08740}{arXiv:2201.08740} (2022), 1-27.
 

\bibitem {DoGi05} I. Doust, T. A. Gillespie,  
Schur multiplier projections on the von Neumann-Schatten classes.  J. Operator Theory 53 (2005), 251-272.

\bibitem {Edwards1977}
R.~E. Edwards and G.~I. Gaudry.
\newblock {\em Littlewood-{P}aley and multiplier theory}.
\newblock Ergebnisse der Mathematik und ihrer Grenzgebiete, Band 90.
  Springer-Verlag, Berlin-New York, 1977.

\bibitem {Grafa08}
L.~Grafakos.
\newblock {\em Classical Fourier analysis}, volume 249 of {\em Graduate Texts
  in Mathematics}.
\newblock Springer, New York, second edition, 2008.

\bibitem {HaStSz10} U. Haagerup, T. Steenstrup, and R. Szwarc, Schur multipliers and spherical functions on homogeneous trees, Internat. J. Math. 21 (2010),  1337-1382.

\bibitem {Ha99} A. Harcharras, Fourier analysis, Schur multipliers on $S_p$ and non-commutative $\Lambda(p)$-sets, Studia Math. 137(1999), 203-260. 

\bibitem {Hytoenen2016}
T.~Hyt\"{o}nen, J.~van Neerven, M.~Veraar, and L.~Weis.
\newblock {\em Analysis in Banach spaces. Vol. I. Martingales and
  Littlewood-Paley theory}, volume~63 of Results in Mathematics and Related Areas. 3rd Series. A Series of Modern
  Surveys in Mathematics. Springer, Cham, 2016.



\bibitem  {HM07} U. Haagerup, M. Musat, On the best constants in noncommutative Khintchine-type inequalities. 
J. Funct. Anal. 250 (2007), 588-624.
 
 

\bibitem  {HW08}T. Hyt\"onen, L. Weis, On the necessity of property ($\alpha$) for some vector-valued multiplier theorems. Arch. Math. (Basel) 90 (2008), 44-52.

\bibitem  {JMP14} M. Junge, T. Mei and J. Parcet, Smooth Fourier multipliers on group von Neumann algebras.
Geom. Funct. Anal. 24 (2014), 1913-1980.

\bibitem {LaSa11}
V.~Lafforgue and M.~De~la Salle,
\newblock Noncommutative {$L^p$}-spaces without the completely bounded
  approximation property.
\newblock {\em Duke Math. J.}, 160(2011), 71--116.

\bibitem {Lan98}G. Lancien,  
Counterexamples concerning sectorial operators. (English summary)
Arch. Math. (Basel) 71 (1998), 388-398. 

\bibitem {LaSa18} T. de Laat and M. de la Salle, Approximation properties for noncommutative $L^p$ of high rank
lattices and nonembeddability of expanders. J. Reine Angew. Math. 737 (2018), 46-69.

\bibitem {L86} F. Lust-Piquard, In\'egalit\'es de Khintchine dans $C^p (1<p<\infty)$, C. R. Acad. Sci. Paris 303 (1986) 289-292. 

\bibitem {LP91} F. Lust-Piquard, G. Pisier, Noncommutative Khintchine and Paley inequalities, Ark. Mat. 29 (1991) 241-260. 


\bibitem{Me07}T. Mei, Operator valued Hardy spaces. Mem. Amer. Math. Soc. 188 (2007), 64 pp.

\bibitem{MRX22} T. Mei, E. Ricard, Q. Xu,  A Mikhlin multiplier theory for free groups and amalgamated free products of von Neumann algebras,  Adv. Math. 403 (2022), Paper No. 108394, 32 pp.

\bibitem{MeXu07} T. Mei, Q. Xu, Rubio de Francia's Littlewood-Paley inequality for operator-valued functions,
{\em https://arxiv.org/abs/0705.1948v1}

 \bibitem {NeRi11}S. Neuwirth, E. Ricard, Transfer of Fourier multipliers into Schur multipliers and sumsets in a discrete group. Canad. J. Math. 63 (2011), 1161-1187. 

\bibitem {PaRiSa22}J. Parcet, E. Ricard, M. de la Salle,   Fourier multipliers in $SL_n({\Bbb R})$. Duke Math. J. 171 (2022), 1235-1297. 


\bibitem{Pi98}G. Pisier, Non-commutative vector valued $L^p$-spaces and completely p-summing maps. Ast\'erisque 247(1998), 131 pp. 

\bibitem{Pi01} G. Pisier,  Similarity problems and completely bounded maps, Second, expanded edition. Includes the solution to ``The Halmos problem'', Lecture Notes in Mathematics, 1618. Springer-Verlag, Berlin, (2001). 198 pp.

\bibitem {PR14} G. Pisier, \'E. Ricard, The noncommutative Khintchine inequalities for $0<p<1$. J. Inst. Math. Jussieu,  16 (2017), 1103-1123.  


\bibitem {PST15} D. Potapov, F. Sukochev and A. Tomskova, On the Arazy conjecture concerning Schur multipliers
on Schatten ideals. Adv. Math. 268 (2015), 404-422.

\bibitem {PST17} D. Potapov, F. Sukochev and A. Tomskova,   Multilinear Schur multipliers and Schatten properties of operator Taylor remainders. Adv. Math. 320 (2017), 1063-1098. 

\bibitem{Ruzhansky2010}
M.~Ruzhansky and V.~Turunen,
\newblock {\em Pseudo-differential operators and symmetries}, volume~2 of {\em
  Pseudo-Differential Operators. Theory and Applications}.
\newblock Birkh\"{a}user Verlag, Basel, 2010.
\newblock Background analysis and advanced topics.


\bibitem{TaWr01}T. Tao, J. Wright,  
Endpoint multiplier theorems of Marcinkiewicz type,
Rev. Mat. Iberoamericana 17 (2001), no. 3, 521-558. 

\bibitem{We01} L. Weis,  Operator-valued Fourier multiplier theorems and maximal $L^p$-regularity, (English summary)
Math. Ann. 319 (2001), no. 4, 735-758.

\end{thebibliography}
\bibliographystyle{abbrv}

\bigskip

{\it \noindent Chian Yeong Chuah,\\
  Department of Mathematics, \\
The Ohio State University,\\
231 West 18th Avenue,\\
Columbus, OH 43210-1174, USA.\\
 ORCid:0000-0003-3776-6555\\
Email:{chuah.21@osu.edu}\\
https://math.osu.edu/people/chuah.21	
 
\medskip

\noindent Zhen-Chuan Liu,\\
Departamento de Matem\'aticas,\\
 Universidad Aut\'onoma de Madrid, \\
 C/ Francisco Tom\'as y Valiente 
\\
28049 Madrid, Spain\\
 ORCid: 0000-0002-6092-5473\\
Email:{liu.zhenchuan@uam.es}
	
 \medskip
 
\noindent Tao Mei, \\Department of Mathematics\\
Baylor University\\
1301 S University Parks Dr, Waco, TX 76798, USA.\\
 ORCid: 0000-0001-6191-6184\\
Email:{tao\_mei@baylor.edu}}\\
https://math.artsandsciences.baylor.edu/person/tao-mei-phd

\end{document}